\documentclass[11pt,oneside]{amsart}
\usepackage{xcolor}
\definecolor{darkblue}{rgb}{0,0,0.6}
\usepackage[colorlinks,citecolor=darkblue,linkcolor=darkblue,urlcolor=darkblue]{hyperref}
\usepackage[left=3cm,right=3cm,top=3cm,bottom=3cm]{geometry}
\usepackage{microtype}

\usepackage{amsmath,amssymb,tikz,tikz-cd}
\usepackage{graphicx} 

\usepackage{etoolbox}
\newcommand{\zerodisplayskips}{%
  \setlength{\abovedisplayskip}{2pt}%
  \setlength{\belowdisplayskip}{2pt}%
  \setlength{\abovedisplayshortskip}{0pt}%
  \setlength{\belowdisplayshortskip}{0pt}}
\appto{\normalsize}{\zerodisplayskips}
\appto{\small}{\zerodisplayskips}
\appto{\footnotesize}{\zerodisplayskips}

\newcommand{\A}{\mathbb{A}}
\newcommand{\AAA}{\mathbb{A}^1}

\providecommand{\GW}{\mathrm{GW}}

\providecommand{\Spec}{\mathrm{Spec}}

\RequirePackage{pict2e,picture}

\makeatletter
\DeclareRobustCommand{\DDelta}{{\mathpalette\bb@Delta\relax}}
\newcommand{\bb@Delta}[2]{%
  \begingroup
  \sbox\z@{$\m@th#1\Delta$}%
  \dimendef\Dht=6 \dimendef\Dwd=8
  \setlength{\Dwd}{\wd\z@}%
  \setlength{\Dht}{\ht\z@}%
  \begin{picture}(\Dwd,\Dht)
  \put(0,0){$\m@th#1\Delta$}
  \put(.42\Dwd,.7\Dht){\line(10,-26){.25\Dht}}
  \end{picture}%
  \endgroup
}

\RequirePackage{mathtools}

\providecommand{\xto}[1]{\xrightarrow{#1}}

\newtheorem{lemma}{Lemma}
\newtheorem{theorem}{Theorem}

\newtheorem{proposition}{Proposition}

\newtheorem{remark}{Remark}
\newcommand{\puncturedaffine}[1]{\A^n\hspace{-0.1em}\smallsetminus\{0\}}

\renewcommand{\P}{\mathbb{P}}

\title{There is no Cazanave's Theorem for punctured affine space}
\author{Thomas Brazelton and William Hornslien}
\date{August 2024}

\begin{document}
\begin{abstract}
    In his thesis, Cazanave proved that the set of naive $\A^1$-homotopy classes of endomorphisms of the projective line admits a monoid structure whose group completion is genuine $\A^1$-homotopy classes of endomorphisms of the projective line. In this very short note we show that such a statement is never true for punctured affine space $\puncturedaffine{n}$ for $n\ge 2$ .
\end{abstract}

\maketitle

\textbf{Assumption}: We work over a base field $k$ of characteristic $\ne 2$.

A foundational theorem of Morel states that the set of $\A^1$-homotopy classes of endomorphisms of the projective line is isomorphic as a ring with $\GW(k)\times_{k^\times} k^\times/(k^\times)^2$ \cite[Theorem 7.36]{Morel}. The genuine homotopy classes emerge from a localization of the category of ($\infty$-categorical) presheaves on smooth $k$-schemes, however one can consider a weaker notion of homotopy, namely identifying two maps $f,g \colon X \to Y$ if there is a map $X \times \A^1_k \to Y$ restricting to~$f$ and $g$ at times $0,1\in \A^1_k$.\footnote{This notion dates back to Gersten and  Karoubi--Villamayor  \cite{ Gersten, KarVilla}. It was called an \textit{elementary homotopy} in \cite{MorelVoevodsky99}.} This is called \textit{naive~$\A^1$-homotopy}, and we denote by naive (resp. genuine) homotopy classes of maps~$[X,Y]^\mathrm{N}$ (resp.~$[X,Y]^{\A^1}$). In general there is a map~$[X,Y]^\mathrm{N} \to [X,Y]^{\A^1}$ but it fails to be a bijection in general.

Cazanave, in his PhD thesis, proved the remarkable result that naive endomorphisms of the projective line $[\P^1,\P^1]^\mathrm{N}$ admits a monoid structure, and the natural map
\begin{align*}
    [\P^1,\P^1]^\mathrm{N} \to [\P^1,\P^1]^{\A^1}
\end{align*}
is a group completion \cite[Proposition 3.23]{Caz}. We show that an analogous result cannot be true for the motivic spheres $\puncturedaffine{n}$ for $n\ge 2$.

Morphisms of punctured affine space $\puncturedaffine{n} \to \puncturedaffine{n}$ are given by tuples $f=(f_1, \ldots, f_n)$ of polynomials in $n$ variables, and these come in two flavors --- those for which $f(0) \ne 0$, and those for which $f(0) = 0$.

\begin{proposition}\label{prop:unimodularity-in-local-algebras}
If $f = (f_1, \ldots, f_n)$ is an endomorphism of punctured affine space, then the ideal $\left\langle f_1, \ldots, f_n \right\rangle \trianglelefteq k[x_1, \ldots, x_n]$ becomes a unimodular row after inverting $x_i$ for any~$1\le i \le n$.
\end{proposition}
\begin{proof} Since $f$ is an endomorphism of punctured affine space, we have that its vanishing locus (which could be empty), is contained in the set containing the origin. By Nullstellensatz this implies that
\begin{align*}
    \left\langle x_1, \ldots, x_n \right\rangle \subseteq \sqrt{\left\langle f_1, \ldots, f_n \right\rangle}.
\end{align*}
Inverting $x_i$ on either side of the equality implies that $1$ is contained in $\left\langle f_1, \ldots, f_n \right\rangle$. 
\end{proof}

We can now ask whether $\left\langle f_1, \ldots, f_n \right\rangle$ is unimodular in the polynomial algebra $k[x_1, \ldots, x_n]$ before inverting any $x_i$. Whether this is true of false has the following consequences.

\begin{lemma}\label{lem:unimodular}
Let $f =(f_1, \ldots, f_n) \colon \puncturedaffine{n} \to \puncturedaffine{n}$ be an endomorphism of punctured affine space.
\begin{enumerate}
    \item If $(f_1, \ldots, f_n)$ is a unimodular row in $k[x_1, \ldots, x_n]$, then $f$ is naively $\A^1$-homotopic to a constant map.
    \item If $(f_1, \ldots, f_n)$ is not a unimodular row in $k[x_1, \ldots, x_n]$, then the local algebra \begin{equation*}
        \frac{k[x_1, \ldots, x_n]_{(x_1, \ldots, x_n)}}{\langle f_1, \ldots, f_n \rangle}
    \end{equation*} is finite length. In the terminology of \cite{KW} this implies that $f$, considered as an endomorphism of affine space, has an \textit{isolated zero at the origin}.
\end{enumerate}
\end{lemma}
\begin{proof} For the first statement, if we suppose $(f_1, \ldots, f_n)$ is a unimodular row in $k[x_1, \ldots, x_n]$, then $f$ extends to a map~${\tilde{f} \colon \mathbb{A}^n \to \puncturedaffine{n}}$. By the Quillen--Suslin theorem, all algebraic vector bundles on affine space are trivial. It follows that the unimodular row is naively homotopy equivalent to a constant map (see \cite[§XXI.3]{Lang}).

On the other hand, if $(f_1, \ldots, f_n)$ is not unimodular in $k[x_1, \ldots, x_n]$, it is still unimodular after inverting $x_i$ for each $i$ by Proposition~\ref{prop:unimodularity-in-local-algebras}. In particular, this implies that there is some~$d_i\in \mathbb{Z}_{\ge0}$ for which
\begin{align*}
    x_i^{d_i} \in \left\langle f_1, \ldots, f_n \right\rangle \trianglelefteq k[x_1, \ldots, x_n].
\end{align*}
This implies that the local algebra $k[x_1, \ldots, x_n]_{(x_1, \ldots, x_n)}/\left\langle f_1, \ldots, f_n \right\rangle$ is finite-dimensional.
\end{proof}

We can now prove the following theorem.

\begin{theorem} For $n\ge 2$, there is no monoid structure on $\left[ \puncturedaffine{n}, \puncturedaffine{n} \right]^\mathrm{N}$ which makes
\begin{align*}
    \left[ \puncturedaffine{n}, \puncturedaffine{n} \right]^\mathrm{N} \to \left[ \puncturedaffine{n},\puncturedaffine{n} \right]^{\A^1} \cong \GW(k)
\end{align*}
into a monoid homomorphism (hence it can never be a group completion).
\end{theorem}
\begin{proof} Since every endomorphism of punctured affine space extends to an endomorphism of affine space, we obtain an induced map on the homotopy cofiber which makes the diagram commute
\[ \begin{tikzcd}
    \puncturedaffine{n}\rar[hook]\dar["f" left] & \A^n\rar\dar  & \frac{\A^n}{\puncturedaffine{n}}\dar[dashed]\dar["\Sigma_{S^1}f" right]\\
    \puncturedaffine{n}\rar[hook] & \A^n\rar & \frac{\A^n}{\puncturedaffine{n}}.
\end{tikzcd} \]
The rightmost map is the $S^1$-suspension of $f$. If $f$ is a unimodular row, it is naively $\A^1$-homotopic to a constant map, so without loss of generality we assume $f$ is not a unimodular row, which implies it has an isolated zero at the origin by Lemma~\ref{lem:unimodular}. In particular since there is a group isomorphism ${\left[ \frac{\A^n}{\puncturedaffine{n}}, \frac{\A^n}{\puncturedaffine{n}} \right]^{\A^1}\cong \GW(k)}$ via Morel's local Brouwer degree at the origin, and we are in the stable range, we conclude that the $\A^1$-degree of $\puncturedaffine{n} \xto{f} \puncturedaffine{n}$ is equal to the local $\A^1$-Brouwer degree of $f$ at the origin. Since $f$ has an isolated zero at the origin, we conclude by \cite{KW} that $\deg_0^{\A^1}(f)$ is an EKL form.

Observe that $\left\langle 1 \right\rangle\in \GW(k)$ is the $\A^1$-Brouwer degree of the identity morphism on~$\puncturedaffine{n}$. If $\left[ \puncturedaffine{n}, \puncturedaffine{n} \right]^\mathrm{N}$ admitted a monoid structure, then $2 \left\langle 1 \right\rangle$ would be representable by an endomorphism of $\puncturedaffine{n}$, and hence would be the local $\A^1$-Brouwer degree of an endomorphism of affine space at the origin. However since EKL forms of rank $\ge2$ must contain a hyperbolic form by a theorem of Quick, Strand, and Wilson \cite[Theorem 2.2]{QSW}, we conclude that no such endomorphism can exist. 
\end{proof}

\begin{remark}
It is still possible that there is a monoid structure on a \emph{subset} of the naive homotopy classes $\left[ \puncturedaffine{n}, \puncturedaffine{n} \right]^\mathrm{N}$ that group completes to $\GW(k)$. For example, Quick, Strand, and Wilson show that for $u\in k^\times$ the quadratic forms $\mathbb{H}$ and $\mathbb{H}+ \langle u \rangle $ are representable by endomorphisms of $\A^n$. A monoid generated by these elements would group complete to~$\GW(k)$.
\end{remark}

The story would have been different if $\puncturedaffine{n}$ was affine scheme for $n \geq 2$.  The set~$[\Spec(A), \puncturedaffine{n}]^\mathrm{N}$ can be identified with unimodular rows of length $n$ in the ring $A$. There are several ways to endow this set with a group structure. Van der Kallen \cite{VANDERKALLEN1983363} used weak Mennicke symbols to construct a group structure, while Asok and Fasel \cite{AsokFasel} have used~$\AAA$-homotopy theory and the fact that $\puncturedaffine{n}$ is an $h$-group in the $\AAA$-homotopy category. Lerbet \cite{LERBET2024109415} proved that the two group structures agree. It is however crucial that the domain is affine to end up in the world of unimodular rows, as the main theorem demonstrates.

\subsection*{Acknowledgements}

We thank  Aravind Asok and Marc Levine for helpful comments and feedback. The first named author is supported by an NSF Postdoctoral Research Fellowship (DMS-2303242).

\bibliographystyle{alpha}
\bibliography{references}
\end{document}